\documentclass{amsart}
\usepackage{amsfonts,amssymb,amsmath,amsthm}
\usepackage{url}
\usepackage{enumerate}

\newtheorem{theorem}{Theorem}[section]
\newtheorem{lemma}[theorem]{Lemma}
\newtheorem{corollary}[theorem]{Corollary}
\theoremstyle{definition}
\newtheorem{example}[theorem]{Example}
\numberwithin{equation}{section}

\newcommand{\abs}[1]{|#1|}
\newcommand{\abslarge}[1]{\left\lvert#1\right\rvert}
\newcommand{\norm}[1]{\lVert#1\rVert}
\newcommand{\normps}[1]{\lVert#1\rVert_\Delta}
\newcommand{\wstar}{weak$^\ast$\ }
\newcommand{\iin}{\,\mathord\in\,}
\newcommand{\M}{\mathfrak{m}}
\newcommand{\N}{\mathfrak{n}}
\newcommand{\Ub}{\mathsf{U_b}}
\newcommand{\LUC}{\mathsf{LUC}}
\newcommand{\MeasSymb}{{\mathsf{M}}}
\newcommand{\UMeas}{\mathsf{\MeasSymb_u}}
\newcommand{\tMeas}{\mathsf{\MeasSymb_t}}
\newcommand{\Mol}{\mathsf{Mol}}
\newcommand{\fset}{\mathcal{F}}     
\newcommand{\BLipb}{\mathsf{BLip_b}}
\newcommand{\real}{\mathbb{R}}
\newcommand{\conv}{\star}               
\newcommand{\bigsep}{\mbox{\large $\;\mid$\;}}
\newcommand{\sect}[1]{\setminus_{#1}}
\newcommand{\pmass}{\partial}           
\newcommand{\rmd}{\,\mathrm{d}}
\newcommand{\compln}[1]{\widehat{#1}}

\newcommand{\compactn}{\compln{\mathsf{p}}}

\newcommand{\RUC}{\mathsf{RUC}}
\newcommand{\ellinfty}{\ell_\infty}

\begin{document}

\title{Continuity of convolution and SIN groups}
\author{Jan Pachl and Juris Stepr\={a}ns}

\begin{abstract}
Let the measure algebra of a topological group $G$ be equipped with
the topology of uniform convergence on bounded right uniformly equicontinuous sets of functions.
Convolution is separately continuous on the measure algebra,
and it is jointly continuous if and only if $G$ has the SIN property.
On the larger space $\LUC(G)^\ast$ which includes the measure algebra,
convolution is also jointly continuous if and only if the group has the SIN property,
but not separately continuous for many non-SIN groups.
\end{abstract}

\maketitle


\section{Introduction}
    \label{sec:intro}

Throughout the paper we assume that topological groups are Hausdorff,
linear spaces are over the field $\real$ of real numbers,
and functions are real-valued.
Our results hold also when scalars are the complex numbers, with essentially the same proofs.

When $G$ is a topological group,
the set of all continuous right-invariant pseudometrics on $G$
induces the topology of $G$ and its right uniformity~\cite[sec.3.2]{Pachl2013usm}
\cite[7.4]{Roelcke1981ust}.
In what follows, we denote by $G$ not only $G$ with its topology but also $G$ with its
right uniformity.
Since we do not consider other uniform structures on $G$, this convention
will not lead to any confusion.

A pseudometric on $G$ is \emph{bi-invariant} iff it is both left- and right-invariant.
A topological group $G$ is a \emph{SIN group},
or has the \emph{SIN property}, iff its topology (equivalently, its right uniformity)
is induced by the set of all continuous bi-invariant pseudometrics~\cite[7.12]{Roelcke1981ust}.

The space $\LUC(G)=\Ub(G)$ of bounded uniformly continuous functions on $G$ has a prominent role
in abstract harmonic analysis.
It is a Banach space with the sup norm.
Its dual $\LUC(G)^\ast$ is a Banach algebra in which the multiplication is the
\emph{convolution operation}~$\conv$, defined as follows.
When $\varphi$ is an expression with several parameters,
$\sect{x}\varphi$ denotes $\varphi$ as a function of $x$.
Define
\begin{align*}
{\N\bullet} f (x) & :=\N(\sect{y} f(xy))
   \quad \text{ for }\N\iin\LUC(G)^\ast, f\iin \LUC(G), x\iin G     \\
\M\conv\N (f) & := \M( {\N\bullet} f )
   \quad \text{ for } \M,\N\iin\LUC(G)^\ast, f\iin \LUC(G)
\end{align*}
Here $(\N,f)\mapsto{\N\bullet} f$ is the canonical left action of $\LUC(G)^\ast$ on $\LUC(G)$.

We identify every finite Radon measure $\mu$ on $G$ with the functional $\M\iin\LUC(G)^\ast$
for which $\M(f) = \int f \rmd \mu$, $f\iin\LUC(G)$.
That way the space $\tMeas(G)$ of finite Radon (a.k.a. tight) measures on $G$ is identified
with a subspace of $\LUC(G)^\ast$.
With convolution, this is the \emph{measure algebra} of $G$, often denoted simply $M(G)$.

Along with the norm topology, another topology on $\LUC(G)^\ast$ and $\tMeas(G)$ commonly considered
is the \wstar topology $w(\LUC(G)^\ast,\LUC(G))$.
Questions about \emph{separate} \wstar continuity of convolution on $\LUC(G)^\ast$
lead to the problem of characterizing the \wstar topological centre of $\LUC(G)^\ast$
and of the LUC compactification of $G$ ---
see~\cite{Ferri2007tca}, \cite{Lau1986cam}, \cite{Lau1995tcc} and Chapter~9 of~\cite{Pachl2013usm}.
\emph{Joint} \wstar continuity of convolution on $\LUC(G)^\ast$
was studied by Salmi~\cite{Salmi2010jcm},
who showed that convolution need not be jointly \wstar continuous
even on bounded subsets of $\tMeas(G)$.

Here we consider the UEB topology on the space $\LUC(G)^\ast$.
This topology, finer than the \wstar topology,
arises naturally in the study of continuity properties of convolution.
When restricted to the space $\tMeas(G)$,
the UEB topology and the \wstar topology $w(\tMeas(G),\LUC(G))$
are closely related:
It follows from general results in Chapter~6 of~\cite{Pachl2013usm}
that these two topologies on $\tMeas(G)$ have the same dual $\LUC(G)$
and the same compact sets (hence the same convergent sequences),
and they coincide on the positive cone of $\tMeas(G)$.

The UEB topology may be defined independently of the group structure of $G$,
for a general uniform space; for the details of the general theory we refer the reader
to~\cite{Pachl2013usm}.
In our current setting of the right uniformity on a topological group $G$,
the UEB topology is defined as follows.
As in~\cite{Pachl2013usm}, for a continuous right-invariant pseudometric $\Delta$ on $G$
and $\M\iin\LUC(G)^\ast$ let
\begin{align*}
\BLipb(\Delta) & := \{ f \colon G \to [-1,1] \bigsep \abs{f(x)-f(y)} \leq \Delta(x,y) \;
\text{ for all } x,y \iin G \} \\
\normps{\M} & := \sup \{ \M(f) \mid f\iin\BLipb(\Delta) \}
\end{align*}
The UEB topology on $\LUC(G)^\ast$ is the locally convex topology defined by the seminorms
$\normps{\cdot}$ where $\Delta$ runs through continuous right-invariant pseudometrics on $G$.

In~\cite{Neufang2015uem} the UEB topology is defined as the topology of uniform convergence
on equi-LUC subsets of $\LUC(G)$.
That definition is equivalent to the one given here,
since by Lemma~3.3 in~\cite{Pachl2013usm} for every equi-LUC set $\fset\subseteq\LUC(G)$
there are $r\iin\real$ and a continuous right-invariant pseudometric $\Delta$ on $G$ such that
$\fset \subseteq r \BLipb(\Delta)$.

When the group $G$ is locally compact
and $\tMeas(G)$ is identified with the algebra of right multipliers of $\mathsf{L}_1(G)$,
the UEB topology on $\tMeas(G)$ coincides
with the right multiplier topology~\cite[Th.3.3]{Neufang2015uem}.
If $G$ is discrete then $\LUC(G)=\ell_\infty(G)$
and the UEB topology on $\LUC(G)^\ast$ is simply its norm topology.
If $G$ is compact then $\LUC(G)$ is the space of continuous functions on $G$ and
the UEB topology is the topology of uniform convergence on norm-compact subsets of $\LUC(G)$.

When the group $G$ is metrizable by a right-invariant metric $\Delta$,
the seminorm $\normps{\cdot}$ on $\LUC(G)^\ast$ is a particular case of
the Kantorovich--Rubinshte{\u\i}n norm, which has many uses
in topological measure theory and in the theory of optimal transport~\cite[8.3]{Bogachev2007mt}
\cite[6.2]{Villani2009oto}.
In this case the topology of $\normps{\cdot}$
coincides with the UEB topology on bounded subsets of $\LUC(G)^\ast$
\cite[sec.5.4]{Pachl2013usm}
but typically not on the whole space $\LUC(G)^\ast$.
As we show in section~\ref{sec:jointc},
when considered on the whole space $\LUC(G)^\ast$ or even $\tMeas(G)$,
convolution behaves better in the UEB topology than in the $\normps{\cdot}$ topology.

Our results in this paper complement those in\cite{Neufang2015uem}.
By Corollary~4.6 and Theorem~4.8 in~\cite{Neufang2015uem},
convolution is jointly UEB continuous on bounded subsets $\LUC(G)^\ast$ when $G$ is a SIN group,
and jointly UEB continuous on the whole space $\LUC(G)^\ast$ when $G$ is a locally compact SIN group.
Our main result (Theorem~\ref{th:main} in section~\ref{sec:jointc}) states that convolution
is jointly UEB continuous on $\LUC(G)^\ast$
if and only if it is jointly UEB continuous on $\tMeas(G)$
if and only if $G$ is a SIN group.
In section~\ref{sec:sepc} we prove that convolution is separately UEB continuous on $\tMeas(G)$
for every topological group $G$,
but not separately continuous on $\LUC(G)^\ast$ for many non-SIN groups.

For locally compact groups, Lau and Pym~\cite{Lau1995tcc} established
the connection between the SIN property and the \wstar continuity of multiplication
in the LUC compactification.
Corollary~\ref{cor:disLC} in section~\ref{sec:sepc} extends one of their results
to a larger class of topological groups.


\section{Preliminaries}

In this section we establish several properties of SIN groups that are needed in the proof
of the main theorem in section~\ref{sec:jointc}.

We specialize the notation of~\cite{Pachl2013usm}, where it is used for functions and measures
on general uniform spaces, to the case of a topological group $G$.
For every $x\iin G$ we denote by $\pmass(x)$ the \emph{point mass} at $x$,
the functional in $\LUC(G)^\ast$ defined by $\pmass(x)(f)=f(x)$ for $f\iin\LUC(G)$.
$\Mol(G)\subseteq\LUC(G)^\ast$ is the space of \emph{molecular measures};
that is, finite linear combinations of point masses.
Obviously $\Mol(G)\subseteq\tMeas(G)$.
For the molecular measure of the special form $\M=\pmass(x)-\pmass(y)$, $x,y\iin G$,
and for any continuous right-invariant pseudometric $\Delta$ on $G$ we have
$\normps{\M}= \min(2,\Delta(x,y))$,
by Lemma~5.12 in~\cite{Pachl2013usm}.

The UEB closure of $\Mol(G)$ in $\LUC(G)^\ast$ is the space $\UMeas(G)\supseteq\tMeas(G)$
of \emph{uniform measures} on the uniform space $G$.
In this paper we do not deal with the space $\UMeas(G)$;
we only point out where a result that we prove for $\tMeas(G)$ holds more generally for $\UMeas(G)$.
The reader is referred to~\cite{Pachl2013usm} for the theory of uniform measures.

We start with a characterization of SIN groups which is one part of~\cite[2.17]{Roelcke1981ust}.

\begin{lemma}
    \label{lem:SINinner}
A topological group $G$ with identity element $e$ is a SIN group if and only if
for every neighbourhood $U$ of $e$ there exists a neighbourhood $V$ of $e$
such that $xVx^{-1} \subseteq U$ for all $x\iin G$.
\qed
\end{lemma}

\begin{lemma}
    \label{lem:upper}
Let $G$ be a SIN group and $\Delta$ a bounded continuous right-invariant pseudometric on $G$.
Then there is a continuous bi-invariant pseudometric $\Theta$ on $G$ such that
$\Theta\geq\Delta$.
\end{lemma}

\begin{proof}
The proof mimics that of Lemma~3.3 in~\cite{Pachl2013usm}.
It is enough to consider the case $\Delta\leq 1$.
As $G$ is a SIN group,
there are continuous bi-invariant pseudometrics $\Theta_j$ for $j=0,1,\dotsc$, such that
\[
\forall x,y \iin S \;\; [ \; \Theta_j (x,y) < 1 \;
\Rightarrow \; \Delta(x,y) < \frac{1}{2^{j+1}} \; ] .
\]
Define $\Theta$ by
\[
\Theta (x,y) := \sum_{j=0}^\infty \frac{1}{2^{j}} \; \min(\Theta_j (x,y),1) .
\]
If $x,y\iin X$ and $j$ are such that $\Theta (x,y) < 1 / 2^{j} $
then $\Theta_j (x,y) < 1$, whence $\Delta(x,y) < 1 / 2^{j+1}$.
It follows that $\Theta\geq\Delta$.
\end{proof}

\begin{corollary}
    \label{cor:biinv}
Let $G$ be a SIN group.
Then the UEB topology on $\LUC(G)^\ast$ is defined by the seminorms $\normps{\cdot}$
where $\Delta$ runs through continuous bi-invariant pseudometrics on $G$.
\qed
\end{corollary}

If $\Delta$ is a continuous or left- or right-invariant pseudometric on $G$,
then so is the pseudometric $\sqrt{\Delta}$ defined by $\sqrt{\Delta}(x,y):=\sqrt{\Delta(x,y)}$
for $x,y\iin G$.

In the sequel we deal with functions of the form ${f}/{\sqrt{\norm{f}}}$ where
$f\iin\LUC(G)$. To simplify the notation, we adopt the convention that
${f}/{\sqrt{\norm{f}}}=f$ when $f$ is identically 0.

\begin{lemma}
    \label{lem:sqroot}
Let $\Delta$ be a pseudometric on a set $G$.
Then ${f}/{\sqrt{\norm{f}}}\iin\BLipb(2\sqrt{\Delta})$
for every $f\iin\BLipb(\Delta)$.
\end{lemma}

\begin{proof}
Take any $x,y\iin G$, and consider two cases: \\
If $\norm{f}\leq\Delta(x,y)$ then
$\abslarge{{f(x)}/{\sqrt{\norm{f}}}},
     \abslarge{{f(y)}/{\sqrt{\norm{f}}}}\leq\sqrt{\Delta(x,y)}$,
hence
\[
\abslarge{\frac{f(x)}{\sqrt{\norm{f}}} - \frac{f(y)}{\sqrt{\norm{f}}}} \leq 2\sqrt{\Delta(x,y)} \;\;.
\]
If $\norm{f}>\Delta(x,y)>0$ then
\[
\abslarge{\frac{f(x)}{\sqrt{\norm{f}}} - \frac{f(y)}{\sqrt{\norm{f}}}}
\leq \frac{\abs{f(x)-f(y)}}{\sqrt{\Delta(x,y)}}
\leq \sqrt{\Delta(x,y)} \;\;.                               \qedhere
\]
\end{proof}

The following lemma is a key ingredient in the proof of Theorem~\ref{th:main}.

\begin{lemma}
    \label{lem:sqrootest}
Let $G$ be a topological group,
$\M, \N \iin \LUC(G)^\ast$,
and let $\Delta$ be a continuous bi-invariant pseudometric on $G$.
Then
\[
\normps{\M\conv\N} \leq \sqrt{2} \;\norm{\M}_{\sqrt{\Delta}} \;\norm{\N}_{2\sqrt{\Delta}} \;\;.
\]
\end{lemma}

\begin{proof}
Take any $f\iin\BLipb(\Delta)$.
As $\Delta$ is left-invariant, we have $\sect{z} f(xz) \iin \BLipb(\Delta)$ for every $x\iin G$,
and $\norm{{\N\bullet}f} \leq \normps{\N}$.
Now $\BLipb(\Delta)\subseteq\BLipb(\sqrt{2\Delta})\subseteq\BLipb(2\sqrt{\Delta})$
because $\sqrt{2t}\geq t$ for $0\leq t \leq 2$,
and thus
$\normps{\cdot}\leq\norm{\cdot}_{\sqrt{2\Delta}}\leq\norm{\cdot}_{2\sqrt{\Delta}}$.
It follows that
\begin{equation}
    \label{eq:RnNorm}
\norm{{\N\bullet}f} \leq \normps{\N} \leq \norm{\N}_{2\sqrt{\Delta}} \;\;.
\end{equation}
For $x,y\iin G$ we have $g:=\frac{1}{2}\sect{z}(f(xz)-f(yz)) \iin \BLipb(\Delta)$,
hence ${g}/{\sqrt{\norm{g}}}\iin\BLipb(2\sqrt{\Delta})$ by Lemma~\ref{lem:sqroot}.
Moreover $2\norm{g} \leq \Delta(x,y)$ because $\Delta$ is right-invariant, so that
\begin{equation}
    \label{eq:RnLip}
\begin{split}
\abs{{\N\bullet} f(x) - {\N\bullet} f(y)}
= 2 \abs{\N(g)}
& = 2 \sqrt{\norm{g}} \; \abslarge{ \N \left(\frac{g}{\sqrt{\norm{g}}} \right) } \\
& \leq \sqrt{2} \sqrt{\Delta(x,y)}  \;\norm{\N}_{2\sqrt{\Delta}} \;\;.
\end{split}
\end{equation}
Putting~(\ref{eq:RnNorm}) and~(\ref{eq:RnLip}) together, we get
${\N\bullet}f \iin \sqrt{2} \norm{\N}_{2\sqrt{\Delta}} \BLipb(\sqrt{\Delta})$.
Hence
\[
\abs{\M\conv\N (f)}
= \abs{\M({\N\bullet}f)}
\leq \sqrt{2} \;\norm{\M}_{\sqrt{\Delta}} \;\norm{\N}_{2\sqrt{\Delta}} \;\; .
\qedhere
\]
\end{proof}


\section{Joint UEB continuity}
    \label{sec:jointc}

For any topological group $G$
the operation $\conv$ is jointly UEB continuous on bounded subsets of $\tMeas(G)$
\cite[4.5]{Neufang2015uem};
in fact, even on bounded subsets of $\UMeas(G)$
\cite[Cor.9.36]{Pachl2013usm}.
However, as we shall see in this section,
convolution need not be jointly UEB continuous on the whole space $\tMeas(G)$.

The UEB topology is defined by certain seminorms $\normps{\cdot}$.
As a warm-up exercise, consider the continuity with respect to a single such seminorm:
Let $G$ be a metrizable topological group whose topology is defined
by a right-invariant metric $\Delta$.
As we pointed out in the introduction, the topology of the norm $\normps{\cdot}$ coincides with
the UEB topology on bounded subsets of $\LUC(G)^\ast$.
Hence $\conv$ is jointly $\normps{\cdot}$ continuous on bounded subsets of $\tMeas(G)$.
However, $\conv$ is not jointly $\normps{\cdot}$ continuous
on the whole space $\tMeas(G)$ or even $\Mol(G)$ for $G=\real$:

\begin{example}
Let $G$ be the additive group $\real$ with the usual metric $\Delta(x,y)=\abs{x-y}$.
For $j=1,2,\dotsc$ let $\M_j := \N_j := j\left(\pmass(1/j^2)-\pmass(0)\right)$
and $f_j(x) := \min(1,\abs{x-(1/j^2)})$ for $x\iin\real$.
Then $f_j\iin\BLipb(\Delta)$ and
\begin{align*}
\M_j \conv \N_j & = j^2 \left( \pmass(2/j^2) - 2 \pmass(1/j^2) + \pmass(0) \right)   \\
\normps{\M_j \conv \N_j} & \geq  \M_j \conv \N_j (f_j) = 2
\end{align*}
but $\lim_j \,\normps{\M_j} = \lim_j \,\normps{\N_j} = 0$.

Note that although the sequence $\{\M_j\}_j$ converges in the norm $\normps{\cdot}$,
it does not converge in the UEB topology;
in fact, $\norm{\M_j}_{\sqrt{\Delta}}=1$ for all $j$.
\qed
\end{example}

Next we shall see that the situation changes when we move from the topology defined by a single
seminorm $\normps{\cdot}$ to the topology defined by all such seminorms,
i.e. the UEB topology.

\begin{theorem}
    \label{th:main}
The following properties of a topological group $G$ are equivalent:
\begin{enumerate}[(i)]
\item
Convolution is jointly UEB continuous on $\LUC(G)^\ast$.
\item
Convolution is jointly UEB continuous on $\tMeas(G)$.
\item
Convolution is jointly UEB continuous on $\Mol(G)$.
\item
$G$ is a SIN group.
\end{enumerate}
\end{theorem}

\begin{proof}
Obviously (i)$\Rightarrow$(ii)$\Rightarrow$(iii).

To prove (iii)$\Rightarrow$(iv), assume that convolution is jointly UEB continuous on $\Mol(G)$.
Take any neighbourhood $U$ of the identity element $e$.
There is a continuous right-invariant pseudometric $\Theta$ such that
$\{z\iin G \mid \Theta(z,e)<1 \} \subseteq U$.
By the UEB continuity there are a continuous right-invariant pseudometric $\Delta$
and $\varepsilon>0$ such that
if $\M,\N\iin\Mol(G)$, $\normps{\M}, \normps{\N}\leq\varepsilon$ then
$\norm{\M\conv\N}_\Theta < 1$.
To conclude that $G$ is a SIN group, in view of Lemma~\ref{lem:SINinner} it is enough
to show that $xVx^{-1} \subseteq U$ for all $x\iin G$, where
$
V:=\{v\iin G \mid \Delta(v,e) < \varepsilon^2 \}
$.
To that end, take any $x\iin G$ and $v\iin V$ and define
\begin{align*}
\M & := \varepsilon \,\pmass(x)    \\
\N & := ( \pmass(v) - \pmass(e) ) / \varepsilon
\end{align*}
Then
$\norm{\M}_\Delta = \varepsilon$ and
$\norm{\N}_\Delta = \min(2,\Delta(v,e)) / \varepsilon < \varepsilon$,
hence
\[
\min(2,\Theta(xv,x)) = \norm{\pmass(xv) -\pmass(x)}_\Theta
= \norm{\M \conv \N}_\Theta < 1
\]
and therefore $\Theta(xvx^{-1},e) = \Theta(xv,x) < 1$ and $xvx^{-1}\iin U$.
That completes the proof of (iii)$\Rightarrow$(iv).

To prove (iv)$\Rightarrow$(i), assume that $G$ is a SIN group.
Take any continuous bi-invariant pseudometric $\Delta$ on $G$.
By Lemma~\ref{lem:sqrootest},
if $\M,\M_0,\N,\N_0\iin\LUC(G)^\ast$ are such that
$\norm{\M-\M_0}_{\sqrt{\Delta}}<\varepsilon$ and
$\norm{\N-\N_0}_{2\sqrt{\Delta}}<\varepsilon$ then
\begin{align*}
\normps{\M\conv\N - \M_0 \conv \N_0}
& \leq \normps{(\M-\M_0)\conv\N} + \normps{\M_0\conv(\N - \N_0)}
  \\
& \leq \sqrt{2} \;\varepsilon \;\norm{\N}_{2\sqrt{\Delta}}
  + \sqrt{2} \;\norm{\M_0}_{\sqrt{\Delta}} \;\varepsilon
  \\
& \leq \sqrt{2} \;\varepsilon
  \left( \varepsilon + \norm{\N_0}_{2\sqrt{\Delta}} + \norm{\M_0}_{\sqrt{\Delta}} \right)
\end{align*}
which along with Corollary~\ref{cor:biinv}
proves that $\conv$ is jointly UEB continuous at $(\M_0,\N_0)$.
\end{proof}


\section{Separate UEB continuity}
    \label{sec:sepc}

By Theorem~\ref{th:main}, convolution is jointly UEB continuous on $\LUC(G)^\ast$,
and therefore also separately UEB continuous, whenever $G$ is a SIN group.
On the other hand, as we explain at the end of this section, there are topological groups $G$
for which convolution is not separately UEB continuous on $\LUC(G)^\ast$.
Nevertheless, we now prove that convolution is separately UEB continuous on $\tMeas(G)$
for every topological group $G$.
The same proof may be used to show that convolution is separately UEB continuous even on $\UMeas(G)$.

\begin{lemma}
    \label{lem:tMeasUEB}
Let $G$ be a topological group, $\M\iin\tMeas(G)$,
and let $\Delta$ be a continuous right-invariant pseudometric on $G$.
Then there exists a continuous right-invariant pseudometric $\Delta_\M$ such that
$\sect{y} \M(\sect{x} f(xy)) \iin \norm{\M} \BLipb(\Delta_\M)$
for every $f\in\BLipb(\Delta)$.
\end{lemma}

\begin{proof}
Evidently $\norm{\sect{y} \M(\sect{x} f(xy))} \leq \norm{\M}$ for every $f\in\BLipb(\Delta)$.
To prove that the function $\sect{y} \M(\sect{x} f(xy))$ is Lipschitz for a suitable $\Delta_\M$,
first note that if $\M=\sum_j c_j \M_j$, $\M_j\iin\LUC(G)^\ast$, is a finite linear combination
such that \[
\abs{\M_j(\sect{x} f(xy)) - \M_j(\sect{x} f(xz))} \leq \Delta_j(y,z)
\]
for every $j$ and $y,z\iin G$,
then
\[
\abs{\M(\sect{x} f(xy)) - \M(\sect{x} f(xz))} \leq \Delta^\prime(y,z)
\]
where $\Delta^\prime = \sum_j \abs{c_j} \Delta_j$.
Thus it is enough to prove the lemma assuming that $\M\geq 0$.

We may also assume that $\Delta\leq 2$,
as replacing $\Delta$ by $\min(\Delta,2)$ does not change $\BLipb(\Delta)$.
For $\M\geq 0$, $\M\neq 0$, and $\Delta\leq 2$,
define $\Delta_\M$ by
\[
\Delta_\M(y,z):=\M(\sect{x} \Delta(xy,xz)) / \norm{\M} \;\text{ for } \; y,z\iin G.
\]
Clearly $\Delta_\M$ is a right-invariant pseudometric.
To see that it is continuous, first apply the estimate
\[
\abs{\Delta(xy,x)-\Delta(wy,w)} \leq \Delta(xy,wy)+\Delta(x,w) = 2 \Delta(x,w)
\]
which shows that $\sect{x} \Delta(xy,x) \iin 2 \BLipb(\Delta)$ for every $y\iin G$.
Since $\M$ is continuous on $2\BLipb(\Delta)$ in the topology of pointwise convergence,
it follows that $\Delta_\M(y,e)$ is a continuous function of $y$ on $G$.

For any $f\iin\BLipb(\Delta)$ we have
\begin{align*}
\abs{\M(\sect{x} f(xy)) - \M(\sect{x} f(xz))}
& \leq \M \left( \sect{x} \abs{ f(xy)-f(xz) } \right)    \\
& \leq \M \left( \sect{x} \Delta(xy,xz) \right)
= \norm{\M} \Delta_\M (y,z)
\end{align*}
for $y,z\iin G$.
\end{proof}

\begin{theorem}
For every topological group $G$
convolution is separately UEB continuous on $\tMeas(G)$.
\end{theorem}

\begin{proof}
For every $\N\iin\LUC(G)^\ast$ the mapping $\M\mapsto\M\conv\N$ is UEB continuous
--- this is a special case of~\cite[Cor.9.21]{Pachl2013usm}.

For $\M\iin\tMeas(G)$ and $\N\iin\LUC(G)^\ast$ we may reverse the order of applying $\M$ and $\N$
in the definition of convolution:
\[
\M\conv\N (f)=\N \left(\sect{y}\M(\sect{x} f(xy)) \right)
\;\text{ for }\; f\iin \LUC(G).
\]
This is a consequence of a variant of Fubini's theorem;
see~\cite[sec.9.4]{Pachl2013usm} for a proof and discussion.

The UEB continuity of the mapping $\N\mapsto\M\conv\N$ for every $\M\iin\tMeas(G)$
now follows from Lemma~\ref{lem:tMeasUEB}.
\end{proof}

In analogy with the commonly studied \wstar topological centre of $\LUC(G)^\ast$,
we may also consider its UEB topological centre $\Lambda_{UEB}$,
the set of those $\M\iin\LUC(G)^\ast$ for which the mapping
$\N \mapsto \M\conv\N$ is UEB continuous on $\LUC(G)^\ast$.
Then $\Lambda_{UEB}=\LUC(G)^\ast$ for every SIN group $G$ by Theorem~\ref{th:main}.
Example~4.7 in~\cite{Neufang2015uem} (which is also Example~9.39 in~\cite{Pachl2013usm})
shows that $\Lambda_{UEB}\neq\LUC(G)^\ast$
when $G$ is the group of homeomorphisms of the interval $[0,1]$ onto itself
with the topology of uniform convergence.
Next we shall show that in fact $\Lambda_{UEB}\neq\LUC(G)^\ast$ for every topological group $G$
that contains a non-SIN subgroup that is locally compact or metrizable.

For any topological group $G$ denote by $\RUC(G)$ the space of those
bounded continuous functions $f$ on $G$
for which the mapping $x\mapsto \sect{y} f(yx)$ is continuous from $G$ to
the space $\ellinfty(G)$ with the sup norm.
In other words, $\RUC(G)$ is the space of bounded left uniformly continuous functions on $G$.

Note that $g\iin\LUC(G)$ if and only if $\sect{x} g(x^{-1})\iin\RUC(G)$.
Thus $\LUC(G)=\RUC(G)$ if and only if $\LUC(G)\subseteq\RUC(G)$.
It is a long-standing open problem whether every topological group $G$
such that $\LUC(G)=\RUC(G)$ is a SIN group.
Bouziad and Troallic~\cite{Bouziad2007pau} survey a number of partial answers,
including the one in the next lemma, which follows from a more general result
of Protasov~\cite{Protasov1991fbg}.

\begin{lemma}
    \label{lem:FSIN}
Let $G$ be a topological group that is locally compact or metrizable and such that
$\LUC(G)\subseteq\RUC(G)$.
Then $G$ is a SIN group.
\qed
\end{lemma}

As in~\cite[sec.6.5]{Pachl2013usm}, when each element $x$ of a topological group $G$
is identified with the point mass $\pmass(x)\iin\LUC(G)^\ast$
and $\LUC(G)^\ast$ is equipped with its \wstar topology,
we obtain topological embeddings $G\subseteq\compln{G}\subseteq G^\LUC \subseteq \LUC(G)^\ast$.
Here $\compln{G}$ is the completion of $G$ (with its right uniformity) and
$G^\LUC=\compactn G$ is its uniform compactification.
The embedding $G\subseteq\compln{G}$ is not only topological but uniform as well.
Both $\compln{G}$ and $G^\LUC$ are subsemigroups of $\LUC(G)^\ast$ with the convolution operation.

The following theorem will be applied in two cases:
When $G$ is locally compact or completely metrizable, we let $S=G$.
When $G$ is merely metrizable, we let $S=\compln{G}$.

\begin{theorem}
    \label{th:contSIN}
Let $G$ be a topological group that is locally compact or metrizable.
Let $S$ be a subsemigroup of $G^\LUC$ such that
\begin{enumerate}[(a)]
\item
$G\subseteq S$,
\item
the topology of $S$ is locally compact or completely metrizable, and
\item
    \label{th:contSIN:itemc}
for every $\M\iin G^\LUC$
the mapping $x \mapsto \M\conv x$ from $S$ to $G^\LUC$ is continuous.
\end{enumerate}
Then $G$ is a SIN group.
\end{theorem}

The main argument in the following proof is used in the proof of~\cite[4.4.5]{Berglund1989aos}.

\begin{proof}
Take any $f\iin\LUC(G)$.
Define $\varphi\colon G^\LUC \times G \to \real$ by
$\varphi(\M,x) := \M\conv x (f)$ for $\M\iin G^\LUC$, $x\iin G$.

From the definition of $\conv$, for every $\N\iin\LUC(G)^\ast$
the mapping $\M \mapsto \M \conv \N$ is \wstar continuous on $\LUC(G)^\ast$.
That along with (\ref{th:contSIN:itemc}) implies that the convolution operation
is separately continuous on the product $G^\LUC \times S$,
therefore jointly continuous on $G^\LUC \times G$ by~\cite[1.4.2]{Berglund1989aos}.

It follows that $\varphi$ is jointly continuous on $G^\LUC \times G$.
Then by~\cite[B.3]{Berglund1989aos} the mapping $x\mapsto \sect{\M} \varphi(\M,x)$ is continuous
from $G$ to $\ellinfty(G^\LUC)$ with the sup norm.
Hence the mapping $x\mapsto \sect{\M} \varphi(\M,x) \upharpoonright G$ is continuous
from $G$ to $\ellinfty(G)$ with the sup norm.
But $\varphi(y,x)=f(yx)$ for $x,y\iin G$,
and we get $f\iin\RUC(G)$ by the definition of $\RUC(G)$.
That proves $\LUC(G)\subseteq\RUC(G)$.
Using Lemma~\ref{lem:FSIN} we conclude that $G$ is a SIN group.
\end{proof}

For locally compact non-SIN groups the following corollary was proved
by Lau and Pym~\cite[3.1]{Lau1995tcc}.

\begin{corollary}
    \label{cor:disLC}
Let $G$ be a non-SIN group whose topology is locally compact or completely metrizable.
Then there exists $\M\iin G^\LUC$ for which the mapping
$x \mapsto \M\conv x$ from $G$ to $G^\LUC$ is not continuous.
\qed
\end{corollary}

Many infinite-dimensional groups of automorphisms,
such as those discussed by Pestov~\cite{Pestov2006did},
are metrizable by a complete metric and not SIN.
This includes the groups of autohomeomorphisms of the interval $[0,1]$
and of the Cantor set $2^{\aleph_0}$ with the topology of uniform convergence,
groups of automorphisms of many Fra\"{i}ss\'{e} structures with the topology of pointwise convergence,
and the unitary group of an infinite-dimensional Hilbert space.

\begin{corollary}
    \label{cor:disMetr}
Let $G$ be a metrizable non-SIN group.
Then there exists $\M\iin G^\LUC$ for which the mapping
$x \mapsto \M\conv x$ from $\compln{G}$ to $G^\LUC$ is not continuous.
\end{corollary}

\begin{proof}
Apply Theorem~\ref{th:contSIN} with $S=\compln{G}$, which of course is completely metrizable.
\end{proof}

By~\cite[Cor.6.13]{Pachl2013usm} the UEB and \wstar topologies coincide on $\compln{G}$.
That together with the two corollaries shows that
for any non-SIN group $G$ that is locally compact or metrizable
there exists $\M\iin G^\LUC$ for which the mapping
$x \mapsto \M\conv x$ from $\compln{G}$ to $G^\LUC$ is not UEB continuous,
and thus convolution is not separately UEB continuous on $G^\LUC$.

More generally, to exhibit such a discontinuity
it is enough to show that one of the two corollaries applies to a subgroup $H$ of $G$.
Indeed, if $H$ is a topological subgroup of $G$ then $H$ is a uniform subspace of $G$
when both are considered with their right uniformities~\cite[3.24]{Roelcke1981ust}.
Hence $H^\LUC$ is embedded in $G^\LUC$, both topologically and algebraically
(with the convolution operation).
It follows that convolution is not separately UEB continuous on $G^\LUC$
whenever $G$ contains a locally compact or metrizable subgroup that is not SIN.

Thus Corollary~\ref{cor:disMetr} holds for a large class of not necessarily metrizable
non-SIN groups.
We do not know whether it holds for every non-SIN group.


\textit{Acknowledgements.}
Jan Pachl appreciates the supportive environment at the Fields Institute.
Research of Juris Stepr\={a}ns for this paper was partially supported by NSERC of Canada.


\vspace{1cm}\noindent
Jan Pachl \\
Fields Institute  \\
222 College Street \\
Toronto, Ontario  M5T 3J1   \\
Canada   \\

\noindent
Juris Stepr\={a}ns   \\
Department of Mathematics and Statistics    \\
York University   \\
Toronto, Ontario  M3J 1P3   \\
Canada

\end{document}